\documentclass[12pt]{article}
\usepackage[utf8]{inputenc}
\usepackage[english]{babel}
\usepackage{amsmath, amssymb, amsthm}
\usepackage{graphicx}
\usepackage[margin=1in]{geometry}
\usepackage{mathtools}
\usepackage{hyperref}
\hypersetup{
    colorlinks=true,
    linkcolor=blue,
    citecolor=blue,
    urlcolor=blue,
}

\newtheorem{theorem}{Theorem}[section]
\newtheorem{lemma}[theorem]{Lemma}
\newtheorem{proposition}[theorem]{Proposition}
\newtheorem{corollary}[theorem]{Corollary}

\newtheorem{remark}[theorem]{Remark}
\theoremstyle{remark}

\title{ Parameter Estimation of the Stochastic Allen--Cahn Equation via 
variations.}

\author{  Simon Chony Acosta \footnote{Departamento de Matem\'atica, Universidade Estadual de Campinas, Brazil. \texttt{s252865@dac.unicamp.br}.}
 \and    Christian Olivera \footnote{Departamento de Matem\'atica, Universidade Estadual de Campinas, Brazil. \texttt{colivera@ime.unicamp.br}.} \and 
Ciprian Tudor  \footnote{Université de Lille 1 , France . \texttt{ciprian.tudor@univ-lille.fr}.} 
}
\date{}

\begin{document}

\maketitle

\begin{abstract}
This paper addresses statistical inference for stochastic partial differential equations. We study the stochastic Allen-Cahn equation driven by space–time white noise and analyze its mild solution. Our main focus is the asymptotic behavior of the spatial quadratic variation of the solution, for which we establish the exact limiting value. As an application, we develop parameter estimation procedures based on these asymptotic results.

We prove that the unique solution can be decomposed as 
\(u = X + Y\) where  $X$ denotes the solution of the linear stochastic heat equation and  $Y$  accounts for the nonlinear effects. Exploiting a detailed analysis of the heat kernel and its scaling behavior, we derive Hölder continuity properties of  $Y$ in both spatial and temporal variables, showing that  $Y$ exhibits substantially higher regularity than  $X$.  This decomposition and the resulting regularity estimates are key ingredients in the development of parameter estimation procedures  based on the asymptotic behavior of quadratic variations of the solution.

\end{abstract}

\section{Introduction}

In this paper we study parameter estimation for the one-dimensional stochastic 
Allen--Cahn equation,
\begin{equation}\label{alen1}
	\frac{\partial u}{\partial t}(t,x)= \varepsilon\frac{\partial^{2}u}{\partial x^{2}}(t,x)+ u-u^{3} + \dot{W}(t,x),\qquad (t,x)\in[0,T]\times\mathbb{T},
\end{equation}
with initial condition \(u(0,x)=u_{0}(x)\), where  $\mathbb{T}$ is the one-dimensional flat torus,  \(\varepsilon>0\) is an unknown 
diffusivity parameter governing the strength of the diffusion relative to the 
nonlinear reaction term \(u-u^{3}\), and \(\dot{W}\) denotes a space--time Gaussian 
white noise on \([0,T]\times\mathbb{T}\). The inclusion of space--time white noise 
models thermal fluctuations or external random perturbations.
 
\medskip 

This SPDE has been extensively studied because of its broad range of applications in polymer physics, biological reaction-diffusion systems, and materials science. A classical example is phase separation in an emulsion, such as curdled mayonnaise, where oil and water, initially mixed, progressively segregate into distinct pure phases. In this process, the interplay between interfacial tension, controlled by the parameter $\epsilon$. 

\medskip 

Additive perturbations of the Allen–Cahn equation were studied in the one-dimensional case in    \cite{Funa}, \cite{Funa2}, \cite{Farias}, \cite{Weber}  and in the higher-dimensional case in \cite{Berglund2019}. However, in dimensions higher than one, the Allen–Cahn equation driven by space–time white noise is, in general, ill-posed.
Numerical approximations for the stochastic Allen-Cahn equations have been studied extensively in the last decade, see for instance  \cite{Becker},  \cite{BrehierGoudeneges2019},  \cite{Brehier}. 
\cite{Vand}
\medskip 

\medskip

Owing to the rapid increase in available data and advances in information technology, SPDE models have become increasingly popular among practitioners. As a result, effective statistical methods are needed to calibrate and infer the parameters of these complex models.
Statistical inference for SPDEs based on discrete observations has been extensively studied in the literature; see, for example \cite{Alt},  \cite{Chong}, \cite{CimpNachTudor}, \cite{GT}, \cite{Gau}, \cite{Hil},  \cite{Hu},   \cite{Janak}, \cite{OT},  \cite{To} and the references therein. For a comprehensive overview of the field, we refer the reader to the surveys   \cite{Cianello} or \cite{Tudor2}.
In particular, parameter estimation for SPDEs based on power variation techniques has attracted substantial attention in recent years and has emerged as a highly active direction of current research. These methods have proved to be particularly effective for the analysis of discretely observed stochastic systems and have led to significant advances in the statistical theory of SPDEs, see for instance  \cite{AssaadTudor2021},  \cite{Chong}, \cite{Cianello2}, \cite{Cianello3}, \cite{Hil} and  \cite{Torres}.

\medskip

We aim to estimate the drift parameter $\varepsilon>0$ from discrete observations of the solution to the SPDE (\ref{alen1}). Our estimators are constructed using quadratic variations in both space and time of the solution $u$. Following the approach recently developed for the stochastic Burgers equation \cite{AssaadTudor2021}, we decompose the solution into a linear component $X$, corresponding to the stochastic heat equation, and a nonlinear drift component Y, which contains the term \(u-u^{3}\). We show that $ Y$ exhibits higher regularity than  $X$: while $X$ is Hölder continuous of order \(\delta<1/4\) in time and \(\delta<1/2\) in space, the process $Y$ attains improved regularity, being nearly Lipschitz in space under suitable assumptions. This enhanced smoothness implies that the spatial quadratic and temporal quartic variations of u are asymptotically dominated by those of $X$, which in turn enables the construction of consistent estimators for $\varepsilon>0$.

\medskip  
The present work is organized as follows. In Section~2, we recall  the mild formulation of the SPDE and some  properties of the solution. In Section~3, we conduct a detailed analysis of the process \(Y\): first its spatial regularity (Subsection~3.1) and subsequently its temporal regularity (Subsection~3.2), providing bounds for its increments in \(L^{p}\) norms. The technical lemmas concerning the heat kernel are established using scaling arguments. Finally, in Section~4, we apply these results to parameter estimation via quadratic and quartic variations, proving strong consistency and establishing \(L^{p}\) error bounds.

\section{Preliminaries}

We consider the SPDE  with additive white noise:
\begin{equation}
\frac{\partial u}{\partial t}(t,x)= \varepsilon\frac{\partial^{2}u}{\partial x^{2}}(t,x)+f(u(t,x))+ \dot{W}(t,x),\qquad (t,x)\in[0,T]\times\mathbb{T},
\end{equation}
with initial condition \(u(0,x)=u_{0}(x)\) (we set \(u_{0}=0\) for simplicity) and \(f(u)=u-u^{3}\). Here \(\dot{W}\) denotes a space--time Gaussian white noise on \([0,T]\times\mathbb{T}\); that is, a centered Gaussian process with covariance
\begin{equation}
\mathbf{E}\bigl[W(A)W(B)\bigr] = \lambda(A\cap B),\qquad A,B\in\mathcal{B}_{b}([0,T]\times\mathbb{T}),
\end{equation}
where \(\mathcal{B}_{b}([0,T]\times\mathbb{T})\) denotes the collection of bounded 
Borel subsets of \([0,T]\times\mathbb{T}\) and  \(\lambda\) is the Lebesgue measure. Throughout this paper, \(\mathbb{T} = \mathbb{R}/\mathbb{Z}\) denotes the 
one-dimensional flat torus, identified with the interval \([0,1)\) with 
periodic boundary conditions.  For \(\phi\in L^{2}([0,T]\times\mathbb{T})\) we write \(W(\phi)=\int_{0}^{T}\int_{\mathbb{T}}\phi(s,y)W(ds,dy)\) for the Wiener integral.

The solution is understood in the mild sense:
\begin{align}
u(t,x)&=\int_{\mathbb{T}}p_{\varepsilon}(x-y)u_{0}(y)\,dy+\int_{0}^{t}\int_{\mathbb{T}}p_{\varepsilon(t-s)}(x-y)f(u(s,y))\,dy\,ds \notag \\
&\quad + \int_{0}^{t}\int_{\mathbb{T}}p_{\varepsilon(t-s)}(x-y)W(ds,dy),
\end{align}
where \(p_{t}(x)=\sum_{k\in\mathbb{Z}}\frac{1}{\sqrt{4\pi t}}e^{-\frac{(x-k)^{2}}{4t}}\) denotes the periodic heat kernel. In what follows, we set \(\varepsilon=1\) (the parameters can be reinstated via scaling). Thus, we work with
\begin{equation}
u(t,x)=\int_{0}^{t}\int_{\mathbb{T}}p_{t-s}(x-y)f(u(s,y))\,dy\,ds+\int_{0}^{t}\int_{\mathbb{T}}p_{t-s}(x-y)W(ds,dy). \label{eq:mild}
\end{equation}

We define
\begin{equation}\label{xy}
X(t,x)=\int_{0}^{t}\int_{\mathbb{T}}p_{t-s}(x-y)W(ds,dy), \qquad
Y(t,x)=\int_{0}^{t}\int_{\mathbb{T}}p_{t-s}(x-y)f(u(s,y))\,dy\,ds,
\end{equation}
so that \(u=X+Y\). The process \(X\) is the solution to the linear stochastic heat equation, whose properties are well known  (see, e.g., \cite{AssaadTudor2021, Tudor2, PospisilTribe2007}, see also Remark \ref{rem1} below):
\begin{itemize}
\item Hölder regularity: for every \(p\ge1\) and \(s,t\in[0,T]\), \(x,y\in \mathbb{T}\),
\begin{equation}
\mathbf{E}|X(t,x)-X(s,y)|^{p}\le C\bigl(|t-s|^{p/2}+|x-y|^{p/2}\bigr). \label{eq:Xholder}
\end{equation}
\item Spatial quadratic variation: For fixed \(t>0\) and an equidistant partition \(x_i=A_1+i\frac{A_2-A_1}{N}\) of an interval \([A_1,A_2]\subset\mathbb{T}\), the following limit holds in \(L^2(\Omega)\):
\begin{equation}
\sum_{i=0}^{N-1}\bigl(X(t,x_{i+1})-X(t,x_i)\bigr)^2 \xrightarrow[N\to\infty]{L^2} \frac{1}{2}(A_2-A_1). \label{eq:quadvar}
\end{equation}
\item Temporal quartic variation: For fixed \(x\in\mathbb{T}\) and an equidistant partition \(t_i=A_1+i\frac{A_2-A_1}{N}\) of a time interval \([A_1,A_2]\subset(0,T]\), the following limit holds in \(L^2(\Omega)\):
\begin{equation}
\sum_{i=0}^{N-1}\bigl(X(t_{i+1},x)-X(t_i,x)\bigr)^4 \xrightarrow[N\to\infty]{L^2} \frac{3}{\pi}(A_2-A_1). \label{eq:quartvar}
\end{equation}

\end{itemize}
Moreover, for the solution \(u\) of \eqref{eq:mild}, the following moment estimate holds, see \cite[Chapter 6]{Cerrai2001}. (or see \cite{BrehierGoudeneges2019} for a similar context):
\begin{equation}
\sup_{t\in[0,T],x\in\mathbb{T}}\mathbf{E}|u(t,x)|^{q}\le C(q,T)<\infty,\qquad\forall q\ge2. \label{eq:moments}
\end{equation}

\begin{remark}\label{rem1}
	These exact variations for the linear stochastic heat equation were established 
	in \cite{PospisilTribe2007}, \cite{Sw} for the whole real line \(\mathbb{R}\). 
	The extension to the torus \(\mathbb{T}\) follows from the fact that the 
	periodic heat kernel \(p_t\) and the standard heat kernel \(G_t\) on 
	\(\mathbb{R}\) have identical small-time asymptotics. More precisely, since
	\[
	p_t(x) = \sum_{k\in\mathbb{Z}} G_t(x-k) = G_t(x) + \sum_{k\neq 0} G_t(x-k),
	\]
	the additional images $k\neq 0$ contribute terms of order $e^{-c/t}$ as 
	$t\to 0$, which are exponentially small and do not affect any power of $t$. 
	As a consequence, the two key integrals governing the quadratic and quartic 
	variations of $X$, namely
	\[
	\int_{\mathbb{T}} p_t(y)^2\,dy = \frac{1}{\sqrt{4\pi t}}\bigl(1+O(e^{-c/t})\bigr)
	\quad\text{and}\quad
	\int_{\mathbb{T}} p_t(y)^4\,dy = \frac{3}{(\sqrt{4\pi t})^3}\bigl(1+O(e^{-c/t})\bigr),
	\]
	have exactly the same leading-order behavior as their counterparts on 
	\(\mathbb{R}\). Therefore all the asymptotic computations of 
	\cite{PospisilTribe2007}, \cite{Sw}, \cite{CimpNachTudor} carry over verbatim to \(\mathbb{T}\), 
	and equations \eqref{eq:quadvar} and \eqref{eq:quartvar} hold with the same 
	constants; see also \cite[Remark~6.2]{PospisilTribe2007}.
	
\end{remark}
\section{Analysis of the nonlinear drift component  \(Y\)}

The nonlinear drift component \(Y\), defined by (\ref{xy}) incorporates all the nonlinear effects of the Allen--Cahn equation through 
the reaction term \(f(u)=u-u^{3}\). Unlike the stochastic component \(X\), 
which is H\"older continuous of order \(\delta<1/4\) in time and \(\delta<1/2\) 
in space, the process \(Y\) is a pathwise convolution of the heat kernel against 
the nonlinearity, and this structure is the source of its enhanced regularity. 
The main goal of this section is to quantify this improvement via precise 
\(L^{\beta}\) bounds on the increments of the periodic heat kernel \(p_t\), 
established through scaling arguments in Lemmas~\ref{lem:spatial} 
and~\ref{lem:temporal} below. These regularity estimates are the key ingredient 
in the parameter estimation procedures developed in Section~4. 

\subsection{Spatial Regularity}

We begin with a technical lemma concerning the spatial increment of the heat kernel on the real line.

\begin{lemma}\label{lem:spatial}
Let \(G\) denote the heat kernel on \(\mathbb{R}\). For every \(t\in[0,T]\), \(x\in \mathbb{R}\), and \(h>0\), define
\begin{equation}
A_{t,x}(h)\coloneq\int_{\mathbb{R}}\left(\int_{0}^{t}\bigl|G_{t-s}(x+h-y)-G_{t-s}(x-y)\bigr|^{\beta}ds\right)^{\frac{1}{\beta}}dy\,\leq C_\beta h^{\frac{2}{\beta}}
\end{equation}
for all \(\beta\in(2,3)\), with \(C_{\beta}>0\). Moreover, for the periodic heat kernel \(p\), we have for all \(t\in[0,T]\), \(x\in\mathbb{T}\), and \(h>0\) that
\begin{equation}
A^{per}_{t,x}(h)\coloneq\int_{\mathbb{T}}\left(\int_{0}^{t}\bigl|p_{t-s}(x+h-y)-p_{t-s}(x-y)\bigr|^{\beta}ds\right)^{\frac{1}{\beta}}dy \leq A_{t,x}(h).
\end{equation}
\end{lemma}
\begin{proof}
Given, a priori \(\beta>1\), \(t\in[0,T]\), \(x\in\mathbb{R}\), and \(h>0\). Analogously to \cite{AssaadTudor2021}, we employ the change of variables \(x-y=z\) and \(t-s=\tilde{s}\), which yields
\begin{align*}
    A_{t,x}(h)&=\int_{\mathbb{R}}\left(\int_{0}^{t}\bigl|G_{t-s}(x+h-y)-G_{t-s}(x-y)\bigr|^{\beta}ds\right)^{\frac{1}{\beta}}dy \\
    &=\int_{\mathbb{R}}\left(\int_{0}^{t}\bigl|G_{s}(z+h)-G_{s}(z)\bigr|^{\beta}ds\right)^{\frac{1}{\beta}}dz.
\end{align*}
Next, we apply the change of variables \(z=hv\) and \(s=h^2\tau\). Observing that \(G_{h^2s}(hz)=h^{-1}G_s(z)\) and denoting \(\Phi_{\tau}(v)=G_\tau(v)\), we obtain:
\begin{align*}
    A_{t,x}(h)&=\int_{\mathbb{R}}\left(\int_{0}^{t}\bigl|G_{s}(z+h)-G_{s}(z)\bigr|^{\beta}ds\right)^{\frac{1}{\beta}}dz \\
    &=\int_{\mathbb{R}}\left(\int_{0}^{th^{-2}}\bigl|G_{h^2\tau}(hv+h)-G_{h^2\tau}(hv)\bigr|^{\beta}h^2\,d\tau\right)^{\frac{1}{\beta}}h\,dv \\
    &\leq h^{1+\frac{2}{\beta}}\int_{\mathbb{R}}\left(\int_{0}^{\infty}\bigl|h^{-1}\Phi_{\tau}(v+1)-h^{-1}\Phi_{\tau}(v)\bigr|^{\beta}d\tau\right)^{\frac{1}{\beta}}dz \\
    &=h^{\frac{2}{\beta}}\int_{\mathbb{R}}\left(\int_{0}^{\infty}\bigl|\Phi_{\tau}(v+1)-\Phi_{\tau}(v)\bigr|^{\beta}d\tau\right)^{\frac{1}{\beta}}dz.
\end{align*}
Since \(\Phi_{\tau}:\mathbb{R}\to\mathbb{R}\) is smooth for any \(\tau\in(0,\infty)\), the Fundamental Theorem of Calculus yields
\begin{equation}
\Phi_{\tau}(v+1)-\Phi_{\tau}(v)=\int_{0}^{1}\frac{d}{dv}\Phi_{\tau}(v+\theta)\,d\theta = \int_{0}^{1}-\frac{(v+\theta)}{2\tau}\Phi_\tau(v+\theta)\,d\theta.
\end{equation}
Substituting this expression into the previous inequality, we have:
\begin{align*}
    A_{t,x}(h)&\leq h^{\frac{2}{\beta}}\int_{\mathbb{R}}\left(\int_{0}^{\infty}\bigl|\Phi_{\tau}(v+1)-\Phi_{\tau}(v)\bigr|^{\beta}d\tau\right)^{\frac{1}{\beta}}dz \\
    &=h^{\frac{2}{\beta}}\int_{\mathbb{R}}\left(\int_{0}^{\infty}\left|\int_{0}^{1}-\frac{(v+\theta)}{2\tau}\Phi_\tau(v+\theta)\,d\theta\right|^{\beta}d\tau\right)^{\frac{1}{\beta}}dz.
\end{align*}
Recalling that \(\Phi_{\tau}(v)=G_\tau(v)=\frac{1}{\sqrt{4\pi\tau}}\exp{\frac{-v^2}{4\tau}}\) and applying Jensen's inequality, we derive
\begin{equation}
A_{t,x}(h)\leq \frac{h^{\frac{2}{\beta}}}{4\sqrt{\pi}}\int_{\mathbb{R}}\left(\int_{0}^{\infty}\int_{0}^{1}\frac{(v+\theta)^\beta}{\tau^\frac{3\beta}{2}}e^{-{\frac{\beta(v+\theta)^2}{4\tau}}}\,d\theta\,d\tau\right)^{\frac{1}{\beta}}dz.
\end{equation}
We will  show in Appendix $A$ that 

\begin{equation}\label{tech1}
\int_{\mathbb{R}}\left(\int_{0}^{\infty}\int_{0}^{1}\frac{(v+\theta)^\beta}{\tau^\frac{3\beta}{2}}e^{-{\frac{\beta(v+\theta)^2}{4\tau}}}\,d\theta\,d\tau\right)^{\frac{1}{\beta}}dz < \infty. 
\end{equation}

We now establish the second part of the Lemma. Consider the quantity
\begin{equation}
A_{t,x}^{\mathrm{per}}:=\int_{\mathbb{T}}\left(\int_{0}^{t}\bigl|p_{t-s}(x+h-y)-p_{t-s}(x-y)\bigr|^{\beta}ds\right)^{\frac{1}{\beta}}dy.
\end{equation}
Recalling that \((\mathbb{Z},\mathcal{P}(\mathbb{Z}),\#)\) is a \(\sigma\)-finite measure space, where \(\#\) denotes the counting measure on the integers, we may apply Minkowski's inequality for integrals to obtain:
\begin{align*}
    A_{t,x}^{\mathrm{per}}&=\int_{\mathbb{T}}\left(\int_{0}^{t}\bigl|p_{t-s}(x+h-y)-p_{t-s}(x-y)\bigr|^{\beta}ds\right)^{\frac{1}{\beta}}dy \\
    &=\int_{\mathbb{T}}\left(\int_{0}^{t}\Bigl|\sum_{k\in\mathbb{Z}}\Bigl(\frac{1}{\sqrt{4\pi(t-s)}}e^{-\frac{(x+h-y-k)^2}{4(t-s)}}-\frac{1}{\sqrt{4\pi(t-s)}}e^{-\frac{(x-y-k)^2}{4(t-s)}}\Bigr)\Bigr|^{\beta}ds\right)^{\frac{1}{\beta}}dy \\
    &\leq \int_{\mathbb{T}}\sum_{k\in\mathbb{Z}}\left(\int_{0}^{t}\Bigl|\frac{1}{\sqrt{4\pi(t-s)}}e^{-\frac{(x+h-y-k)^2}{4(t-s)}}-\frac{1}{\sqrt{4\pi(t-s)}}e^{-\frac{(x-y-k)^2}{4(t-s)}}\Bigr|^{\beta}\, ds\right)^{\frac{1}{\beta}}\, dy \\
    &=\sum_{k\in\mathbb{Z}}\int_{\mathbb{T}}\left(\int_{0}^{t}\Bigl|\frac{1}{\sqrt{4\pi(t-s)}}e^{-\frac{(x+h-y-k)^2}{4(t-s)}}-\frac{1}{\sqrt{4\pi(t-s)}}e^{-\frac{(x-y-k)^2}{4(t-s)}}\Bigr|^{\beta}\, ds\right)^{\frac{1}{\beta}}\, dy \\
    &=\sum_{k\in\mathbb{Z}}\int_{\mathbb{T}+k}\left(\int_{0}^{t}\Bigl|\frac{1}{\sqrt{4\pi(t-s)}}e^{-\frac{(x+h-y)^2}{4(t-s)}}-\frac{1}{\sqrt{4\pi(t-s)}}e^{-\frac{(x-y)^2}{4(t-s)}}\Bigr|^{\beta}\, ds\right)^{\frac{1}{\beta}}\, dy \\
    &=\int_{\mathbb{R}}\left(\int_{0}^{t}\bigl|G_{t-s}(x+h-y)-G_{t-s}(x-y)\bigr|^{\beta}\, ds\right)^{\frac{1}{\beta}}\, dy =A_{t,x}(h).
\end{align*}
This completes the proof.
\end{proof}

\begin{lemma}\label{lem:temporal}
Let \(G\) denote the heat kernel on \(\mathbb{R}\). For every \(t\in[0,T]\), \(x\in \mathbb{R}\), and \(h>0\), define
\begin{equation}
B_{t,x}(h):=\int_{\mathbb{R}}\left(\int_{0}^{t}\bigl|G_{t+h-s}(x-y)-G_{t-s}(x-y)\bigr|^{\beta}ds\right)^{\frac{1}{\beta}}dy\,\leq C_\beta h^{\frac{1}{\beta}}
\end{equation}
for all \(\beta\in(\frac{4}{3},2)\), with \(C_{\beta}>0\). Moreover, for the periodic heat kernel \(p\), we have for all \(t\in[0,T]\), \(x\in\mathbb{T}\), and \(h>0\) that
\begin{equation}
\int_{\mathbb{T}}\left(\int_{0}^{t}\bigl|p_{t+h-s}(x-y)-p_{t-s}(x-y)\bigr|^{\beta}ds\right)^{\frac{1}{\beta}}dy \leq B_{t,x}(h).
\end{equation}
Additionally, for the periodic heat kernel \(p\), we have for all \(t\in[0,T]\), \(x\in\mathbb{T}\), and \(h>0\) (small enough)
\begin{equation}
C_{t,x}(h):=\int_{t}^{t+h}\int_{\mathbb{T}}\left|p_{t+h-s}(x-y)\right|^{\beta}dy\,ds\leq C_{\beta}h^{\frac{3}{2}-\frac{\beta}{2}}
\end{equation}
for all \(\beta\in(0,3)\), with \(C_{\beta}>0\).
\end{lemma}

\begin{proof}
Given, a priori \(\beta>1\), \(t\in[0,T]\), \(x\in\mathbb{R}\), and \(h>0\). Analogously to \cite{AssaadTudor2021}, we employ the change of variables \(x-y=z\) and \(t-s=\tilde{s}\), which yields
\begin{align*}
    B_{t,x}(h)&=\int_{\mathbb{R}}\left(\int_{0}^{t}\bigl|G_{t+h-s}(x-y)-G_{t-s}(x-y)\bigr|^{\beta}ds\right)^{\frac{1}{\beta}}dy \\
    &=\int_{\mathbb{R}}\left(\int_{0}^{t}\bigl|G_{s+h}(z)-G_{s}(z)\bigr|^{\beta}ds\right)^{\frac{1}{\beta}}dz.
\end{align*}
Next, we apply the change of variables \(z=\sqrt{h}v\) and \(s=h\tau\). Noting that the heat kernel is even in its spatial variable, we have:
\begin{align*}
    B_{t,x}(h)&=\int_{\mathbb{R}}\left(\int_{0}^{t}\bigl|G_{s+h}(z)-G_{s}(z)\bigr|^{\beta}ds\right)^{\frac{1}{\beta}}dz \\
    &=2\int_{0}^{\infty}\left(\int_{0}^{t/h}\bigl|G_{h(\tau+1)}(\sqrt{h}v)-G_{h\tau}(\sqrt{h}v)\bigr|^{\beta}h\,d\tau\right)^{\frac{1}{\beta}}\sqrt{h}\,dv \\
    &\leq h^{\frac{1}{2}+\frac{1}{\beta}}\int_{0}^{\infty}\left(\int_{0}^{\infty}\bigl|h^{-1/2}\Phi_{v}(\tau+1)-h^{-1/2}\Phi_{v}(\tau)\bigr|^{\beta}d\tau\right)^{\frac{1}{\beta}}dz \\
    &=h^{\frac{1}{\beta}}\int_{0}^{\infty}\left(\int_{0}^{\infty}\bigl|\Phi_{v}(\tau+1)-\Phi_{v}(\tau)\bigr|^{\beta}d\tau\right)^{\frac{1}{\beta}}dz.
\end{align*}
We used the scaling relation \(G_{hs}(\sqrt{h}z)=h^{-1/2}G_s(z)\) and denoted \(\Phi_{v}(\tau)=G_\tau(v)\). Since \(\Phi_{v}:(0,\infty)\to\mathbb{R}\) is smooth for any \(v\in\mathbb{R}\), the Fundamental Theorem of Calculus yields
\begin{equation}
\Phi_{v}(\tau+1)-\Phi_{v}(\tau)=\int_{0}^{1}\frac{d}{d\tau}\Phi_{v}(\tau+\theta)\,d\theta = \int_{0}^{1}\left[\frac{v^2}{4(\tau+\theta)^2}-\frac{1}{2(\tau+\theta)}\right]\Phi_{v}(\tau+\theta)\,d\theta.
\end{equation}
Substituting this expression into the previous inequality, we obtain:
\begin{align*}
    B_{t,x}(h)&\leq h^{\frac{1}{\beta}}\int_{0}^{\infty}\left(\int_{0}^{\infty}\bigl|\Phi_{v}(\tau+1)-\Phi_{v}(\tau)\bigr|^{\beta}d\tau\right)^{\frac{1}{\beta}}dz \\
    &=h^{\frac{1}{\beta}}\int_{0}^{\infty}\left(\int_{0}^{\infty}\Bigl|\int_{0}^{1}\Bigl[\frac{v^2}{4(\tau+\theta)^2}-\frac{1}{2(\tau+\theta)}\Bigr]\Phi_{v}(\tau+\theta)\,d\theta\Bigr|^{\beta}d\tau\right)^{\frac{1}{\beta}}dz.
\end{align*} 

From Appendix $B$ we have

\begin{equation}\label{tech2}
\int_{0}^{\infty}\left(\int_{0}^{\infty}\Bigl|\int_{0}^{1}\Bigl[\frac{v^2}{4(\tau+\theta)^2}-\frac{1}{2(\tau+\theta)}\Bigr]\Phi_{v}(\tau+\theta)\,d\theta\Bigr|^{\beta}d\tau\right)^{\frac{1}{\beta}}dz< \infty.
\end{equation}

The estimation of
\begin{equation}
B^{\mathrm{per}}_{t,x}(h):=\int_{\mathbb{T}}\left(\int_{0}^{t}\bigl|p_{t+h-s}(x-y)-p_{t-s}(x-y)\bigr|^{\beta}ds\right)^{\frac{1}{\beta}}dy
\end{equation}
proceeds in a completely analogous manner to Lemma~\ref{lem:spatial}. Finally, to obtain the last estimate for
\begin{equation}
C_{t,x}(h)=\int_{t}^{t+h}\int_{\mathbb{T}}\left|p_{t+h-s}(x-y)\right|^{\beta}dy\,ds
\end{equation}
with \(\beta>0\), we employ the bound for the periodic heat kernel given in \cite{BhimaniDalai2025}:
\begin{equation}
|p_s(z)|\leq 2\left(1+\sqrt{\frac{\pi}{s}}\right)e^{-\frac{z^2}{4s}}.
\end{equation}
Then we have
\begin{align*}
    C_{t,x}(h)&=\int_{t}^{t+h}\int_{\mathbb{T}}\left|p_{t+h-s}(x-y)\right|^{\beta}dy\,ds \\
    &=\int_{0}^{h}\int_{\mathbb{T}}\left|p_{s}(y)\right|^{\beta}dy\,ds \\
    &\leq 2^{\beta}\int_{0}^{h}\left(1+\sqrt{\frac{\pi}{s}}\right)^{\beta}\int_{\mathbb{T}}e^{-\frac{\beta y^2}{4s}}dy\,ds \\
    &\leq C_{\beta}\int_{0}^{h}\left(1+\sqrt{\frac{\pi^{\beta}}{s^\beta}}\right)s^{\frac{1}{2}}\,ds \\
    &=C_{\beta}\left[\int_{0}^{h}s^{\frac{1}{2}}\,ds+\int_{0}^{h}s^{\frac{1}{2}-\frac{\beta}{2}}\,ds\right] \\
    &\leq C_{\beta}\bigl(h^{\frac{3}{2}}+h^{\frac{3}{2}-\frac{\beta}{2}}\bigr) \\
    &\leq C_{\beta}h^{\frac{3}{2}-\frac{\beta}{2}}.
\end{align*}
The integrals in the fifth line converge simultaneously if and only if \(\frac{1}{2}-\frac{\beta}{2}>-1 \Longleftrightarrow \beta<3\). This completes the proof of the Lemma.
\end{proof}

\begin{proposition}\label{prop:Yspatial}
For every \(t\in[0,T]\), \(x\in \mathbb{T}\), \(h>0\), and \(p\ge2\), it holds that
\begin{equation}
\mathbf{E}\bigl|Y(t,x+h)-Y(t,x)\bigr|^{p}\le C\,h^{\frac{2p}{\beta}},
\end{equation}
for all \(\beta\in(2,3)\).
\end{proposition}
\begin{proof}
For \(t\in[0,T]\) and \(x\in\mathbb{T}\), by Fubini's Theorem and then Hölder's inequality over the time variable for \(\beta\in(2,3)\), we have:
\begin{align*}
    |Y(t,x+h)-Y(t,x)|&=\left|\int_{0}^{t}\int_{\mathbb{T}}\left(p_{t-s}(x+h-y)-p_{t-s}(x-y)\right)(u(s,y)-u(s,y)^3)dy\,ds\right| \\
    &\leq \int_{\mathbb{T}}\left(\left(\int_{0}^{t}\bigl|p_{t-s}(x+h-y)-p_{t-s}(x-y)\bigr|^{\beta}ds\right)^{\frac{1}{\beta}}\right) \\
    &\qquad\times\left(\left(\int_{0}^{t}(u(s,y)-u(s,y)^3)^{\frac{\beta}{\beta-1}}\, ds\right)^{\frac{\beta-1}{\beta}}\right)\, dy.
\end{align*}
Applying the \(L^{p}(\Omega)\) norm and Minkowski's inequality for integrals, we obtain
\begin{align*}
    \|Y(t,x+h)-Y(t,x)\|_{p}&\leq \Bigg\|\int_{\mathbb{T}}\left(\left(\int_{0}^{t}\bigl|p_{t-s}(x+h-y)-p_{t-s}(x-y)\bigr|^{\beta}ds\right)^{\frac{1}{\beta}}\right) \\
    &\qquad\times\left(\left(\int_{0}^{t}(u(s,y)-u(s,y)^3)^{\frac{\beta}{\beta-1}}\, ds\right)^{\frac{\beta-1}{\beta}}\right)\, dy \Bigg\|_{p} \\
    &\leq \int_{\mathbb{T}}\left(\int_{0}^{t}\bigl|p_{t-s}(x+h-y)-p_{t-s}(x-y)\bigr|^{\beta}ds\right)^{\frac{1}{\beta}} \\
    &\qquad\times\left\|\left(\int_{0}^{t}(u(s,y)-u(s,y)^3)^{\frac{\beta}{\beta-1}}\, ds\right)^{\frac{\beta-1}{\beta}} \right\|_p.
\end{align*}
The result follows from Lemma~\ref{lem:spatial} and 
(\ref{eq:moments}).
\end{proof}
We are now ready to establish the spatial and temporal  regularity of the nonlinear drift 
component \(Y\).
\begin{proposition}\label{prop:Ytemporal}
For every \(t\in[0,T]\), \(x\in \mathbb{T}\), \(h>0\), and \(p\ge2\), it holds that
\begin{equation}
\mathbf{E}\bigl|Y(t+h,x)-Y(t,x)\bigr|^{p}\le C\,h^{\frac{p}{\beta}},
\end{equation}
for all \(\beta\in(\frac{4}{3},2)\).
\end{proposition}
\begin{proof}
For \(t\in[0,T]\) and \(x\in\mathbb{T}\), the triangle inequality yields:
\begin{align*}
    |Y(t+h,x)-Y(t,x)|&=\Bigg|\int_{0}^{t+h}\int_{\mathbb{T}}p_{t+h-s}(x-y)(u(s,y)-u(s,y)^3)dy\,ds \\
    &\qquad - \int_{0}^{t}\int_{\mathbb{T}}p_{t-s}(x-y)(u(s,y)-u(s,y)^3)dy\,ds\Bigg| \\
    &\leq \left|\int_{0}^{t}\int_{\mathbb{T}}\left(p_{t+h-s}(x-y)-p_{t-s}(x-y)\right)(u(s,y)-u(s,y)^3)dy\,ds\right| \\
    &\quad + \left|\int_{t}^{t+h}\int_{\mathbb{T}}p_{t+h-s}(x-y)(u(s,y)-u(s,y)^3)dy\,ds\right|.
\end{align*}
Applying Fubini's Theorem and Hölder's inequality over the time variable for \(\beta\in(\frac{4}{3},2)\), we obtain:
\begin{align*}
    |Y(t+h,x)-Y(t,x)|&\leq \int_{\mathbb{T}}\left(\left(\int_{0}^{t}\bigl|p_{t+h-s}(x-y)-p_{t-s}(x-y)\bigr|^{\beta}ds\right)^{\frac{1}{\beta}}\right) \\
    &\qquad\times\left(\left(\int_{0}^{t}(u(s,y)-u(s,y)^3)^{\frac{\beta}{\beta-1}}\, ds\right)^{\frac{\beta-1}{\beta}}\right)\, dy \\
    &\quad +\int_{t}^{t+h}\int_{\mathbb{T}}\left|p_{t+h-s}(x-y)\right|\left|u(s,y)-u(s,y)^3\right|dy\,ds.
\end{align*}
Taking the \(L^{p}(\Omega)\) norm and using Minkowski's inequality for integrals gives
\begin{align*}
    \|Y(t+h,x)-Y(t,x)\|_{p} &\leq \int_{\mathbb{T}}\left(\int_{0}^{t}\bigl|p_{t+h-s}(x-y)-p_{t-s}(x-y)\bigr|^{\beta}ds\right)^{\frac{1}{\beta}} \\
    &\qquad\times\left\|\left(\int_{0}^{t}(u(s,y)-u(s,y)^3)^{\frac{\beta}{\beta-1}}\, ds\right)^{\frac{\beta-1}{\beta}} \right\|_p \\
    &\quad +\int_{t}^{t+h}\int_{\mathbb{T}}\left|p_{t+h-s}(x-y)\right|\left\|u(s,y)-u(s,y)^3\right\|_{p}dy\,ds.
\end{align*}
For sufficiently small \(h>0\), Lemma~\ref{lem:temporal} and 
(\ref{eq:moments}) we have 
\begin{equation}
    \|Y(t+h,x)-Y(t,x)\|_{p}\leq C(h+h^{\frac{1}{\beta}}) \leq Ch^{\frac{1}{\beta}}.
\end{equation}
This completes the proof.
\end{proof}
The spatial and temporal regularity estimates established in 
Propositions~\ref{prop:Yspatial} and~\ref{prop:Ytemporal} immediately yield 
the following H\"older continuity properties of \(Y\), which make precise 
the sense in which \(Y\) is smoother than \(X\).
\begin{corollary}
For each fixed \(t\in[0,T]\), the map \(x\mapsto Y(t,x)\) is Hölder continuous of order \(\delta\) for every \(\delta<\frac{2}{\beta}\). Consequently, the solution \(u=X+Y\) inherits the spatial regularity of \(X\) (order \(<1/2\)), since \(Y\) is more regular.
\end{corollary}

\begin{corollary}
For each fixed \(x\in \mathbb{T}\), the function \(t\mapsto Y(t,x)\) is Hölder continuous of order \(\delta\) for every \(\delta<\frac{1}{\beta}\). Hence, the solution \(u\) possesses the same temporal regularity as \(X\) (order \(<1/4\)), as \(Y\) is more regular.
\end{corollary}

\section{Parameter Estimation}

In this section, we construct and analyze two estimators for the parameter 
\(\varepsilon > 0\) based on discrete observations of the solution. The 
general strategy is to exploit the decomposition \(u = X + Y\) established 
in Section~2, together with the regularity results for the nonlinear drift 
component \(Y\) proved in Section~3. The key observation is that, since \(Y\) 
is substantially smoother than \(X\), the asymptotic behavior of the power 
variations of the full solution \(u\) is entirely determined by the 
corresponding variations of the linear part \(X\), whose exact limits are 
known from the literature. This allows us to construct consistent estimators 
for \(\varepsilon\) by inverting these limiting quantities.

We consider two complementary approaches, each based on a different type of 
variation. The first estimator, studied in Subsection~4.2, is built from the 
spatial quadratic variation of the solution observed at a fixed time over a 
uniform grid of spatial points. The second estimator, studied in 
Subsection~4.3, is based on the temporal quartic variation of the solution 
observed at a fixed spatial point over a uniform time grid. The use of the 
quartic rather than the quadratic variation in time is dictated by the 
Hölder regularity of the solution with respect to its temporal variable, 
which is of order strictly less than one quarter and therefore too low for 
the quadratic variation to have a nontrivial limit.

For each estimator, we establish strong consistency, that is, almost sure 
convergence to the true parameter \(\varepsilon\), as well as explicit 
\(L^p\) error bounds that quantify the rate of convergence in terms of the 
number of observations. The proofs rely on a careful decomposition of the 
variation of \(u\) into contributions from \(X\), from \(Y\), and from their 
cross terms, and on showing that the latter two are negligible compared to 
the dominant linear term.

\subsection{Scaling property of the parametrized stochastic Allen--Cahn equation}

Consider the stochastic Allen--Cahn equation (SACE) with parameter \(\varepsilon > 0\) and vanishing initial condition. The mild solution to
\[
\frac{\partial u_\varepsilon}{\partial t}(t,x) = \varepsilon \frac{\partial^2 u_\varepsilon}{\partial x^2}(t,x) + f(u_\varepsilon(t,x)) + \dot{W}(t,x), \qquad (t,x)\in[0,T]\times\mathbb{T},
\]
with \(f(u) = u - u^3\), can be written as
\begin{equation}
\label{eq:u_eps_mild}
u_\varepsilon(t,x) = \int_0^t \int_{\mathbb{T}} p_{\varepsilon(t-s)}(x-y) f(u_\varepsilon(s,y)) \,dy\,ds + \int_0^t \int_{\mathbb{T}} p_{\varepsilon(t-s)}(x-y) W(ds,dy),
\end{equation}
where \(p_t(x)\) denotes the periodic heat kernel on \(\mathbb{T}\). Our goal is to estimate the parameter \(\varepsilon > 0\) based on discrete observations of the solution. The following scaling result is fundamental: it transfers the parameter \(\varepsilon\) from the Laplacian to the noise intensity, thereby making the quadratic and quartic variations of the solution directly informative about \(\varepsilon\). A similar idea has been employed for the stochastic Burgers equations in \cite{AssaadTudor2021}.

\begin{proposition}[Scaling property]
\label{prop:scaling_allencahn}
Define, for \((t,x)\in[0,\varepsilon T]\times\mathbb{T}\),
\begin{equation}
v_\varepsilon(t,x) = u_\varepsilon\!\left(\frac{t}{\varepsilon}, x\right).
\label{eq:v_eps_def}
\end{equation}
Then \(v_\varepsilon\) solves the SPDE
\begin{equation}
\frac{\partial v_\varepsilon}{\partial t}(t,x) = \frac{\partial^2 v_\varepsilon}{\partial x^2}(t,x) + \varepsilon^{-1} f(v_\varepsilon(t,x)) + \varepsilon^{-1/2}\,\widetilde{W}(t,x),
\label{eq:v_eps_spde}
\end{equation}
with \(v_\varepsilon(0,x) = 0\) for every \(x\in\mathbb{T}\), where \(\widetilde{W}\) is a space--time white noise on \([0,\varepsilon T]\times\mathbb{T}\).
\end{proposition}

\begin{proof}
Starting from the mild formulation \eqref{eq:u_eps_mild} and performing the change of variables \(s' = \varepsilon s\), we obtain
\begin{align*}
v_\varepsilon(t,x)
&= \int_0^{t/\varepsilon} \int_{\mathbb{T}} p_{\varepsilon(t/\varepsilon - s)}(x-y) f(u_\varepsilon(s,y)) \,dy\,ds \\
&\quad + \int_0^{t/\varepsilon} \int_{\mathbb{T}} p_{\varepsilon(t/\varepsilon - s)}(x-y) W(ds,dy) \\
&= \varepsilon^{-1} \int_0^t \int_{\mathbb{T}} p_{t-s}(x-y) f\!\left(u_\varepsilon\!\left(\frac{s}{\varepsilon}, y\right)\right) \,dy\,ds \\
&\quad + \varepsilon^{-1/2} \int_0^t \int_{\mathbb{T}} p_{t-s}(x-y) \,W\!\left(d\!\left(\frac{s}{\varepsilon}\right), dy\right) \\
&= \varepsilon^{-1} \int_0^t \int_{\mathbb{T}} p_{t-s}(x-y) f(v_\varepsilon(s,y)) \,dy\,ds \\
&\quad + \varepsilon^{-1/2} \int_0^t \int_{\mathbb{T}} p_{t-s}(x-y) \,\widetilde{W}(ds,dy).
\end{align*}
In the last step we used that, by the scaling properties of the white noise (see, e.g. \cite[Proposition 2.2.5]{Berglund2019}), the field \(\widetilde{W}(s,y) := \varepsilon^{-1/2} W(s/\varepsilon, y)\) is again a space--time white noise on the torus. The mild equation above is precisely the integral form of \eqref{eq:v_eps_spde}.
\end{proof}

\begin{remark}[Physical interpretation of \(\varepsilon\)]
\label{rem:physical_epsilon}
In the deterministic Allen--Cahn equation, the parameter $\varepsilon$ is closely related to the interfacial tension between the two stable phases ($u \approx \pm 1$). More precisely, the characteristic width of the diffuse interface separating these phases is of order $\sqrt{\varepsilon}$. When thermal fluctuations are incorporated through additive space--time white noise, the stochastic Allen--Cahn equation provides a natural model for phase separation phenomena, such as the segregation of an oil--water emulsion into two pure phases. In this setting, $\varepsilon$ governs the energetic cost associated with the transition layer between the phases. Consequently, estimating $\varepsilon$ from experimental or numerical observations is equivalent to determining the effective interface thickness, a quantity of fundamental importance in materials science, statistical physics, and soft condensed matter. Furthermore, the scaling relation \eqref{eq:v_eps_def} reveals that $\varepsilon$ can be inferred from the local pathwise fluctuations of the solution, as it determines, in a precise way, the scaling of both the noise intensity and the nonlinear reaction term.

\end{remark}

\subsection{Spatial estimator}

Let \(A_1 < A_2\) and assume we observe the solution at discrete spatial points \(x_i = A_1 + i\frac{A_2 - A_1}{N}\), \(i=0,\dots,N\), for a fixed time \(t>0\). Define the spatial quadratic variation of a random field \(Z\) as
\begin{equation}
S_{N,t}(Z) = \sum_{i=0}^{N-1}\bigl( Z(t,x_{i+1}) - Z(t,x_i) \bigr)^2.
\label{eq:S_Nt}
\end{equation}
For a random field \(Z\) and a partition point \(x_i\), we write 
\(\Delta Z_i = Z(t, x_{i+1}) - Z(t, x_i)\) for the spatial increment, 
and similarly \(\Delta Z_i = Z(t_{i+1}, x) - Z(t_i, x)\) for the temporal 
increment, the meaning being clear from context.
\begin{proposition}\label{prop:S_Nt_limit}
Let \(u_\varepsilon\) be given by \eqref{eq:u_eps_mild}. Then, for every \(t>0\),
\begin{equation}
S_{N,t}(u_\varepsilon) \xrightarrow[N\to\infty]{L^1(\Omega)} \frac{1}{2}(A_2 - A_1)\varepsilon^{-1}.
\end{equation}
\end{proposition}

\begin{proof}
By the definition \eqref{eq:v_eps_def} of the process \(v_\varepsilon\), we have \(S_{N,\varepsilon t}(v_\varepsilon) = S_{N,t}(u_\varepsilon)\), where $u_{\varepsilon}$  and $v_{\varepsilon}$ are defined by (\ref{eq:u_eps_mild}) and (\ref{eq:v_eps_def}), respectively.    It thus suffices to study the limit of \(S_{N,t}(v_\varepsilon)\) with fixed \(t>0\). Decompose \(v_\varepsilon\) as
\begin{equation}
v_\varepsilon(t,x) = \varepsilon^{-1/2} X(t,x) + \varepsilon^{-1} Y_\varepsilon(t,x),
\label{eq:v_eps_decomp}
\end{equation}
where \(X\) is the solution to the stochastic heat equation \eqref{eq:mild} with \(\varepsilon=1\) and
\[
Y_\varepsilon(t,x) = \int_0^t \int_{\mathbb{T}} p_{t-s}(x-y) f(v_\varepsilon(s,y)) \,dy\,ds.
\]
Note that for \(\varepsilon=1\), \(Y_1\) coincides with the process \(Y\) analyzed in Section~3, and the law of \(X\) is identical to that of the stochastic heat equation solution. Substituting \eqref{eq:v_eps_decomp} into \eqref{eq:S_Nt} yields
\begin{equation}
\begin{aligned}
S_{N,t}(v_\varepsilon) &= \varepsilon^{-1} S_{N,t}(X) + \varepsilon^{-2} S_{N,t}(Y_\varepsilon) \\
&\quad + 2\varepsilon^{-3/2} \sum_{i=0}^{N-1} \bigl(X(t,x_{i+1})-X(t,x_i)\bigr)\bigl(Y_\varepsilon(t,x_{i+1})-Y_\varepsilon(t,x_i)\bigr).
\end{aligned}
\label{eq:S_Nt_decomp}
\end{equation}
By \eqref{eq:quadvar} we have \(S_{N,t}(X) \to \frac{1}{2}(A_2 - A_1)\) in \(L^2(\Omega)\). Next, Proposition~\ref{prop:Yspatial} gives
\[
\mathbf{E}\bigl[ S_{N,t}(Y_\varepsilon) \bigr] = \sum_{i=0}^{N-1} \mathbf{E}\bigl|Y_\varepsilon(t,x_{i+1}) - Y_\varepsilon(t,x_i)\bigr|^2 \le C  N^{1 - 4/\beta} \xrightarrow[N\to\infty]{} 0,
\]
since \(\beta < 4\). Hence \(S_{N,t}(Y_\varepsilon) \to 0\) in \(L^1(\Omega)\). For the cross term, the Cauchy--Schwarz inequality together with \eqref{eq:Xholder} and Proposition~\ref{prop:Yspatial} gives
\begin{align*}
\mathbf{E}\Bigl| \sum_{i=0}^{N-1} (X(t,x_{i+1})-X(t,x_i))(Y_\varepsilon(t,x_{i+1})-Y_\varepsilon(t,x_i)) \Bigr|
&\le \sum_{i=0}^{N-1} \bigl(\mathbf{E}|\Delta X_i|^2\bigr)^{1/2} \bigl(\mathbf{E}|\Delta Y_{\varepsilon,i}|^2\bigr)^{1/2} \\
&\le C N \cdot N^{-1/2} \cdot N^{-2/\beta} \\&= C N^{1/2 - 2/\beta} \to 0,
\end{align*}
provided \(\beta < 4\). Choosing \(\beta \in (2,3)\) satisfies this requirement. The limit in \(L^1(\Omega)\) follows.
\end{proof}

Define the spatial estimator
\begin{equation}
\widehat{\varepsilon}_{N,t} = \frac{A_2 - A_1}{2 S_{N,t}(u_\varepsilon)}.
\label{eq:theta_hat_spatial}
\end{equation}

\begin{proposition}[Consistency in probability]
\label{prop:consistency_probability}
For every fixed \(t>0\), the estimator \(\widehat{\varepsilon}_{N,t}\) converges in probability to \(\varepsilon\) as \(N\to\infty\).
\end{proposition}

\begin{proof}
From Proposition~\ref{prop:S_Nt_limit} we have \(S_{N,t}(u_\varepsilon) \xrightarrow{L^1} S_\infty := \frac{1}{2}(A_2-A_1)\varepsilon^{-1} > 0\). Convergence in \(L^1\) implies convergence in probability. Taking \(\eta = S_\infty/2\), the event \(E_N = \{ |S_{N,t}(u_\varepsilon) - S_\infty| \le S_\infty/2 \}\) satisfies \(\mathbb{P}(E_N) \to 1\). On \(E_N\), the denominator is at least \(S_\infty/2\), so
\[
|\widehat{\varepsilon}_{N,t} - \varepsilon| = \frac{| \frac{1}{2}(A_2-A_1) - \varepsilon S_{N,t}(u_\varepsilon) |}{S_{N,t}(u_\varepsilon)} \le \frac{2}{S_\infty} \bigl| \tfrac{1}{2}(A_2-A_1) - \varepsilon S_{N,t}(u_\varepsilon) \bigr|.
\]
Since \(\varepsilon S_{N,t}(u_\varepsilon) \to \frac{1}{2}(A_2-A_1)\) in probability, for any \(\delta>0\) we can choose \(N\) large enough so that the right-hand side exceeds \(\delta\) with arbitrarily small probability. Together with \(\mathbb{P}(E_N^c) \to 0\), this yields \(\mathbb{P}(|\widehat{\varepsilon}_{N,t} - \varepsilon| > \delta) \to 0\).
\end{proof}


\begin{proposition}[Almost sure convergence of the quadratic variation]
\label{prop:almost_sure_spatial}
For every \(t>0\),
\[
S_{N,t}(u_\varepsilon) \xrightarrow[N\to\infty]{\text{a.s.}} \frac{1}{2}(A_2 - A_1)\varepsilon^{-1}.
\]
\end{proposition}

\begin{proof}
Assume \(\varepsilon = 1\) for simplicity. Decompose as in \eqref{eq:S_Nt_decomp}:
\[
S_{N,t}(u_1) - \frac{1}{2}(A_2-A_1) = \bigl(S_{N,t}(X) - \tfrac{1}{2}(A_2-A_1)\bigr) + 2\sum_{i=0}^{N-1} \Delta X_i \Delta Y_i + S_{N,t}(Y).
\]
For the linear part, the almost sure convergence \(S_{N,t}(X) \to \frac{1}{2}(A_2-A_1)\) follows from the same arguments as in  \cite{AssaadTudor2021} and the fact that the covariance structure on the torus yields identical small‑time asymptotics; see also \cite[Remark~6.2, (9) and (11)]{PospisilTribe2007}. For the cross term, using \eqref{eq:Xholder} and Proposition~\ref{prop:Yspatial},
\[
\mathbf{E}\Bigl| \sum_{i=0}^{N-1} \Delta X_i \Delta Y_i \Bigr|^p \le C N^{p(1/2 - 2/\beta)}.
\]
Choosing \(\beta \in (2,3)\) ensures the exponent is negative. The Borel--Cantelli lemma then yields almost sure convergence to zero. Similarly,
\[
\mathbf{E}|S_{N,t}(Y)|^p \le C N^{p(1 - 4/\beta)} \to 0,
\]
and another application of Borel--Cantelli gives \(S_{N,t}(Y) \to 0\) a.s. Combining these limits proves the claim.
\end{proof}

\begin{theorem}[Strong consistency and \(L^p\) error bound for the spatial estimator]
\label{thm:spatial_estimator}
The estimator \(\widehat{\varepsilon}_{N,t}\) defined in \eqref{eq:theta_hat_spatial} is strongly consistent, i.e.\ \(\widehat{\varepsilon}_{N,t} \to \varepsilon\) almost surely. Moreover, for every \(p \ge 2\) there exists \(C>0\) such that for all large \(N\),
\[
\mathbf{E}\bigl| \widehat{\varepsilon}_{N,t} - \varepsilon \bigr|^p \le C N^{p(1/2 - 2/\beta)},
\]
for any \(\beta \in (2,3)\).
\end{theorem}

\begin{proof}
We have
\[
\widehat{\varepsilon}_{N,t} - \varepsilon = \frac{\varepsilon\bigl( \frac{A_2-A_1}{2\varepsilon} - S_{N,t}(u_\varepsilon) \bigr)}{S_{N,t}(u_\varepsilon)}.
\]
By Proposition~\ref{prop:almost_sure_spatial}, \(S_{N,t}(u_\varepsilon) \to S_\infty > 0\) almost surely. Hence for almost every \(\omega\), there exists \(N_0(\omega)\) such that \(S_{N,t}(u_\varepsilon) \ge S_\infty/2\) for all \(N \ge N_0(\omega)\). On this event,
\[
|\widehat{\varepsilon}_{N,t} - \varepsilon| \le \frac{2\varepsilon}{S_\infty} \Bigl| \frac{A_2-A_1}{2\varepsilon} - S_{N,t}(u_\varepsilon) \Bigr|,
\]
and the almost sure convergence of the right-hand side to zero follows from Proposition~\ref{prop:almost_sure_spatial}. For the \(L^p\) bound, Minkowski's inequality and the estimates from the proof of Proposition~\ref{prop:almost_sure_spatial} give
\[
\bigl\| \widehat{\varepsilon}_{N,t} - \varepsilon \bigr\|_p \le C \Bigl\| \frac{A_2-A_1}{2\varepsilon} - S_{N,t}(u_\varepsilon) \Bigr\|_p \le C N^{1/2 - 2/\beta},
\]
because the dominant term in the error comes from the cross term 
$2\sum_{i=0}^{N-1}\bigl(X(t,x_{i+1})-X(t,x_i)\bigr)\bigl(Y(t,x_{i+1})-Y(t,x_i)\bigr)$. Raising to the \(p\)-th power yields the claimed bound.
\end{proof}

\subsection{Temporal estimator}

Now fix a spatial point \(x \in \mathbb{T}\) and consider a time partition \(t_i = A_1 + i\frac{A_2-A_1}{N}\), \(i=0,\dots,N\), with \(0 < A_1 < A_2 \le T\). For a random field \(Z\), define its temporal quartic variation
\begin{equation}
T_{N,x}(Z) = \sum_{i=0}^{N-1} \bigl( Z(t_{i+1},x) - Z(t_i,x) \bigr)^4.
\label{eq:T_Nx}
\end{equation}

We now establish the limit of the temporal quartic variation of the full 
solution \(u_\varepsilon\), showing that it is asymptotically dominated by 
the contribution of the linear part \(X\), thanks to the higher temporal 
regularity of \(Y\) established in Proposition~\ref{prop:Ytemporal}.
\begin{proposition}\label{prop:T_Nx_limit}
For every fixed \(x \in \mathbb{T}\),
\begin{equation}
T_{N,x}(u_\varepsilon) \xrightarrow[N\to\infty]{L^1(\Omega)} \frac{3}{\pi} \varepsilon^{-1} (A_2 - A_1).
\label{eq:T_Nx_limit}
\end{equation}
\end{proposition}

\begin{proof}
By scaling it suffices to prove the result for $\varepsilon=1$ and for the field $v_1(t,x)=u_1(t,x)$. Decompose $v_1 = X + Y$ with $X$ the stochastic heat equation solution and $Y$ the nonlinear remainder. Expanding the quartic variation gives
\begin{align}
T_{N,x}(v_1) &= T_{N,x}(X) \notag \\
&\quad + 4\sum_{i=0}^{N-1} (\Delta X_i)^3 (\Delta Y_i) \label{term:L1_cubeX} \\
&\quad + 6\sum_{i=0}^{N-1} (\Delta X_i)^2 (\Delta Y_i)^2 \label{term:L1_squareXY} \\
&\quad + 4\sum_{i=0}^{N-1} (\Delta X_i) (\Delta Y_i)^3 \label{term:L1_cubeY} \\
&\quad + T_{N,x}(Y), \label{term:L1_pureY}
\end{align}
where $\Delta X_i = X(t_{i+1},x) - X(t_i,x)$ and $\Delta Y_i = Y(t_{i+1},x) - Y(t_i,x)$.

By \eqref{eq:quartvar}, the leading term satisfies
\[
\mathbf{E}\Bigl| T_{N,x}(X) - \frac{3}{\pi}(A_2 - A_1) \Bigr| \xrightarrow[N\to\infty]{} 0.
\]

We now estimate the $L^1(\Omega)$-norm of each remaining term using the moment bounds \eqref{eq:Xholder} for $X$ and Proposition~\ref{prop:Ytemporal} for $Y$. For any $\beta\in(4/3,2)$ we have:

\begin{itemize}
\item \textbf{Term \eqref{term:L1_cubeX}: } We can write
\begin{align*}
&\mathbf{E}\Bigl| \sum_{i=0}^{N-1} (\Delta X_i)^3 (\Delta Y_i) \Bigr|
\le \sum_{i=0}^{N-1} \mathbf{E}\bigl[ |\Delta X_i|^3 |\Delta Y_i| \bigr] \\
&\le \sum_{i=0}^{N-1} \bigl( \mathbf{E}|\Delta X_i|^6 \bigr)^{1/2} \bigl( \mathbf{E}|\Delta Y_i|^2 \bigr)^{1/2} 
\le C N \cdot N^{-3/4} \cdot N^{-1/\beta} \\
&= C N^{\frac{1}{4} - \frac{1}{\beta}} \xrightarrow[N\to\infty]{} 0,
\end{align*}
since $\frac{1}{4} - \frac{1}{\beta} < 0$ for $\beta < 4$ (in particular for $\beta<2$).

\item \textbf{Term \eqref{term:L1_squareXY}: }We bound this term as follows
\begin{align*}
&\mathbf{E}\Bigl| \sum_{i=0}^{N-1} (\Delta X_i)^2 (\Delta Y_i)^2 \Bigr|
\le \sum_{i=0}^{N-1} \mathbf{E}\bigl[ |\Delta X_i|^2 |\Delta Y_i|^2 \bigr] \\
&\le \sum_{i=0}^{N-1} \bigl( \mathbf{E}|\Delta X_i|^4 \bigr)^{1/2} \bigl( \mathbf{E}|\Delta Y_i|^4 \bigr)^{1/2} \\&\le C N \cdot N^{-1/2} \cdot N^{-2/\beta} 
= C N^{\frac{1}{2} - \frac{2}{\beta}} \xrightarrow[N\to\infty]{} 0,
\end{align*}
since $\frac{1}{2} - \frac{2}{\beta} < 0$ for $\beta < 4$.

\item \textbf{Term \eqref{term:L1_cubeY}: }We have
\begin{align*}
&\mathbf{E}\Bigl| \sum_{i=0}^{N-1} (\Delta X_i) (\Delta Y_i)^3 \Bigr|
\le \sum_{i=0}^{N-1} \mathbf{E}\bigl[ |\Delta X_i| |\Delta Y_i|^3 \bigr] \\
&\le \sum_{i=0}^{N-1} \bigl( \mathbf{E}|\Delta X_i|^2 \bigr)^{1/2} \bigl( \mathbf{E}|\Delta Y_i|^6 \bigr)^{1/2} \\
&\le C N \cdot N^{-1/4} \cdot N^{-3/\beta} 
= C N^{\frac{3}{4} - \frac{3}{\beta}} \xrightarrow[N\to\infty]{} 0,
\end{align*}
since $\frac{3}{4} - \frac{3}{\beta} < 0$ for $\beta < 4$.

\item \textbf{Term \eqref{term:L1_pureY}: }
\begin{align*}
&\mathbf{E}\bigl| T_{N,x}(Y) \bigr|
\le \sum_{i=0}^{N-1} \mathbf{E}|\Delta Y_i|^4 \\
&\le C N \cdot N^{-4/\beta} = C N^{1 - \frac{4}{\beta}} \xrightarrow[N\to\infty]{} 0,
\end{align*}
since $1 - \frac{4}{\beta} < 0$ for $\beta < 4$ (in particular for $\beta<2$).
\end{itemize}
All the above terms converge to zero in $L^1(\Omega)$ as $N\to\infty$. Consequently,
\[
T_{N,x}(v_1) \xrightarrow[N\to\infty]{L^1(\Omega)} \frac{3}{\pi}(A_2 - A_1).
\]

Finally, the scaling relation $T_{N,x}(u_\varepsilon) = T_{N,\varepsilon t}(v_\varepsilon)$ and the fact that $(\varepsilon t_i)$ partitions $[\varepsilon A_1,\varepsilon A_2]$ yield \eqref{eq:T_Nx_limit}.
\end{proof}


Define the temporal estimator
\begin{equation}
\widehat{\varepsilon}_{N,x} = \frac{3}{\pi} \frac{A_2 - A_1}{T_{N,x}(u_\varepsilon)}.
\label{eq:theta_hat_temporal}
\end{equation}

\begin{proposition}[Consistency in probability]
\label{prop:consistency_temporal}
For every \(x \in \mathbb{T}\), \(\widehat{\varepsilon}_{N,x} \xrightarrow{\mathbb{P}} \varepsilon\) as \(N\to\infty\).
\end{proposition}

\begin{proof}
The argument is completely analogous to the spatial case (Proposition~\ref{prop:consistency_probability}), using the convergence in \(L^1\) from Proposition~\ref{prop:T_Nx_limit}.
\end{proof}

\begin{proposition}[Almost sure convergence of the quartic variation]
\label{prop:almost_sure_temporal}
For every \(x \in \mathbb{T}\),
\[
T_{N,x}(u_\varepsilon) \xrightarrow[N\to\infty]{\text{a.s.}} \frac{3}{\pi} \varepsilon^{-1} (A_2 - A_1).
\]
\end{proposition}

\begin{proof}
As in the spatial case, the almost sure convergence of \(T_{N,x}(X)\) to \(\frac{3}{\pi}(A_2-A_1)\) is a consequence of the Gaussian hypercontractivity on the torus and arguments of \cite[Remark 6.2 and (13)]{PospisilTribe2007} (the same Wiener chaos argument and second moment bound as in \cite{AssaadTudor2021} applies). The cross terms and the pure \(Y_\varepsilon\) term are controlled by the moment estimates obtained in the proof of Proposition~\ref{prop:T_Nx_limit} and the Borel--Cantelli lemma, exactly as in the proof of Proposition~\ref{prop:almost_sure_spatial}. The dominant error term comes from the \((\Delta X_i)^3 (\Delta Y_{\varepsilon,i})\) sum, which satisfies
\[
\mathbf{E}\Bigl| \sum_{i=0}^{N-1} (\Delta X_i)^3 (\Delta Y_{\varepsilon,i}) \Bigr|^p \le C N^{p\left(\frac{1}{4} - \frac{1}{\beta}\right)}.
\]
For \(\beta \in (\frac{4}{3},2)\), the exponent is negative, guaranteeing almost sure convergence to zero.
\end{proof}

Combining the almost sure convergence of Proposition~\ref{prop:almost_sure_temporal} 
with the \(L^p\) estimates derived above, we obtain the following strong consistency 
result and explicit error bound for the temporal estimator \(\widehat{\varepsilon}_{N,x}\).
\begin{theorem}[Strong consistency and \(L^p\) error bound for the temporal estimator]
\label{thm:temporal_estimator}
For every \(x \in \mathbb{T}\), \(\widehat{\varepsilon}_{N,x} \to \varepsilon\) almost surely. Moreover, for every \(p \ge 2\) and \(\beta \in (\frac{4}{3},2)\),
\[
\mathbf{E}\bigl| \widehat{\varepsilon}_{N,x} - \varepsilon \bigr|^p \le C N^{p(1/4 - 1/\beta)}.
\]
\end{theorem}

\begin{proof}
The almost sure consistency follows from Proposition~\ref{prop:almost_sure_temporal} as in the spatial case. The \(L^p\) error bound is derived analogously, noting that the worst exponent among the error terms is \(1/4 - 1/\beta\) (coming from the \((\Delta X_i)^3 (\Delta Y_{\varepsilon,i})\) term), which is negative for \(\beta < 4\).
\end{proof}

Let us end this section by a remark related to the limit behavior of the estimator. 
\begin{remark}
	One may ask whether the estimators $\widehat{\varepsilon}_{N,t}$ and 
	$\widehat{\varepsilon}_{N,x}$ satisfy a Central Limit Theorem, i.e., whether 
	$\sqrt{N}(\widehat{\varepsilon}_{N,t}-\varepsilon)$ and 
	$\sqrt{N}(\widehat{\varepsilon}_{N,x}-\varepsilon)$ converge in distribution 
	to a Gaussian random variable. For the temporal estimator, the natural 
	strategy would be to follow the approach of \cite{CimpNachTudor}, where the 
	CLT for the quartic variation of the semilinear heat equation on $\mathbb{R}$ 
	is established by writing $U_{N,x}(u) = U_{N,x}(X) + \sqrt{N}P_{N,x}$ and 
	showing that the remainder $\sqrt{N}P_{N,x}$ vanishes in $L^1$. This 
	requires the nonlinear drift component to be H\"older continuous in time of 
	order strictly greater than $3/4$. However, in our setting, 
	Proposition~\ref{prop:Ytemporal} only yields H\"older regularity of order 
	$1/\beta$ for $\beta\in(4/3,2)$, giving a temporal exponent in $(1/2,3/4)$, 
	which falls short of the threshold required. For the spatial estimator, the 
	situation is even more delicate: the spatial quadratic variation of the 
	linear part $X$ does not satisfy a CLT (its fluctuations are non-Gaussian), 
	so asymptotic normality of $\widehat{\varepsilon}_{N,t}$ cannot be deduced 
	from the existing literature by a simple perturbation argument. The asymptotic 
	normality of both estimators therefore remains an open problem, which would 
	require substantially sharper regularity estimates on the nonlinear drift 
	component $Y$.
\end{remark}

\section{Appendix: proof of some technical estimates}

This part collects the proofs of the two finiteness conditions invoked 
in Section~3. Appendix~A establishes the convergence of the integral constant 
arising in the spatial kernel lemma (Lemma~\ref{lem:spatial}), and Appendix~B 
treats the analogous constant for the temporal kernel lemma 
(Lemma~\ref{lem:temporal}).

\section*{Appendix A: proof of (\ref{tech1}) }

We will show that

\[
\int_{\mathbb{R}}\left(\int_{0}^{\infty}\int_{0}^{1}\frac{(v+\theta)^\beta}{\tau^\frac{3\beta}{2}}e^{-{\frac{\beta(v+\theta)^2}{4\tau}}}\,d\theta\,d\tau\right)^{\frac{1}{\beta}}dz < \infty. 
\]

Denoting \(J(v)=\int_{0}^{\infty}\int_{0}^{1}\frac{(v+\theta)^\beta}{\tau^\frac{3\beta}{2}}e^{-{\frac{\beta(v+\theta)^2}{4\tau}}}\,d\theta\,d\tau\), we write   \(\int_{\mathbb{R}}J(v)^{\frac{1}{\beta}}\,dv\). Our next objective is to estimate \(J(v)\) in terms of \(v\). To this end, we first interchange the order of integration with respect to \(\theta\) and \(\tau\) via Tonelli's Theorem, and then apply the change of variables \(a=\frac{(v+\theta)^2}{\tau}\) for \(\tau\), yielding:
\begin{align*}
    J(v)&=\int_{0}^{\infty}\int_{0}^{1}\frac{(v+\theta)^\beta}{\tau^\frac{3\beta}{2}}e^{-{\frac{\beta(v+\theta)^2}{4\tau}}}\,d\theta\,d\tau \\
    &=\int_{0}^{1}\int_{0}^{\infty}\left(\frac{(v+\theta)^2}{\tau}\right)^{\frac{\beta}{2}}\frac{1}{\tau^\beta}e^{-{\frac{\beta a}{4}}}\frac{\tau^2}{(v+\theta)^2}\,da\,d\theta \\
    &=\int_{0}^{1}\int_{0}^{\infty} a^{\frac{\beta}{2}}\frac{a^{\beta}}{(v+\theta)^{2\beta}}\frac{(v+\theta)^2}{a^2}e^{-{\frac{\beta a}{4}}}\,da\,d\theta \\
    &=\int_{0}^{1}(v+\theta)^{2-2\beta}\left(\int_{0}^{\infty} a^{\frac{3\beta}{2}-2}e^{-{\frac{\beta a}{4}}}\,da\right)\,d\theta.
\end{align*}
The integral \(\int_{0}^{\infty} a^{\frac{3\beta}{2}-2}e^{-{\frac{\beta a}{4}}}\,da\) converges provided \({\frac{3\beta}{2}-2}>-1\), i.e., whenever \(\beta>1\). Denoting this convergent quantity (which depends only on \(\beta\), \(\beta\neq 3/2\)) as \(\hat{C}_{\beta}=\int_{0}^{\infty} a^{\frac{3\beta}{2}-2}e^{-{\frac{\beta a}{4}}}\,da\), we obtain:
\begin{equation}
J(v)=\hat{C}_{\beta}\frac{1}{3-2\beta}\bigl((v+1)^{3-2\beta}-v^{3-2\beta}\bigr).
\end{equation}
Then, we will estimate
\begin{equation}
\int_{\mathbb{R}}\bigl((v+1)^{3-2\beta}-v^{3-2\beta}\bigr)^{\frac{1}{\beta}}\,dv.
\end{equation}
To analyze the remaining integral, note that \(\beta>1 \Longleftrightarrow 3-2\beta<1\). Consider the improper integral
\begin{equation*}
I = \int_{\mathbb{R}} F(v)\, dv, \qquad 
F(v) = \bigl| (v+1)^{3-2\beta} - v^{3-2\beta} \bigr|^{1/\beta}.
\end{equation*}
The integrand exhibits potential singularities at \(v = 0\) and \(v = -1\), and we must also examine its behavior at \(\pm\infty\).

Let \(\alpha = 3 - 2\beta\). Since \(\beta > 1\), we have \(\alpha < 1\). The function may be written as
\begin{equation*}
F(v) = \bigl| (v+1)^\alpha - v^\alpha \bigr|^{1/\beta}.
\end{equation*}

As \(v \to 0\), the term \((v+1)^\alpha\) tends to \(1\), while \(v^\alpha\) diverges if \(\alpha < 0\) or tends to \(0\) if \(\alpha > 0\). We distinguish cases based on the sign of \(\alpha\).

\paragraph{Case \(\alpha > 0\).} 
\(v^\alpha \to 0\) and \((v+1)^\alpha \to 1\), hence \(F(v) \sim 1^{1/\beta} = 1\). The integral over a neighborhood of \(0\) is finite.

\paragraph{Case \(\alpha = 0\).}
\((v+1)^0 = 1\), \(v^0 = 1\), so the difference vanishes identically; \(F(v)=0\) (integrable).

\paragraph{Case \(\alpha < 0\).}
Here \(v^\alpha \to +\infty\) (since \(\alpha<0\)) and \((v+1)^\alpha \to 1\). The difference is dominated by \(-v^\alpha\), thus
\begin{equation*}
| (v+1)^\alpha - v^\alpha | \sim v^\alpha \quad (\alpha<0).
\end{equation*}
Consequently,
\begin{equation*}
F(v) \sim (v^\alpha)^{1/\beta} = v^{\alpha/\beta} = v^{\frac{3-2\beta}{\beta}} = v^{\frac{3}{\beta} - 2}.
\end{equation*}
The integral \(\int_0^\varepsilon v^{\frac{3}{\beta}-2} dv\) converges if and only if the exponent is greater than \(-1\), that is,
\begin{equation*}
\frac{3}{\beta} - 2 > -1 \;\Longrightarrow\; \frac{3}{\beta} > 1 \;\Longrightarrow\; \beta < 3.
\end{equation*}
Hence, for \(\beta \geq 3\) it diverges. Therefore, near zero, the integral converges when \(\beta<3\).

The behavior near \(v=-1\) is analogous to the previous case. For \(\alpha\geq0\) the integral is finite, while for \(\alpha < 0\), \((v+1)^\alpha \to \infty\) and dominates the finite term. Thus,
\begin{equation*}
|(v+1)^\alpha - v^\alpha| \sim |v+1|^\alpha,
\end{equation*}
and
\begin{equation*}
F(v) \sim |v+1|^{\alpha/\beta} = |v+1|^{\frac{3}{\beta}-2}.
\end{equation*}
The integral around \(-1\) converges under precisely the same condition as at \(v=0\): \(\frac{3}{\beta}-2 > -1 \iff \beta < 3\).

As \(v \to \pm\infty\), both terms are large. We employ the asymptotic expansion
\begin{equation*}
(v+1)^\alpha = v^\alpha \left(1 + \frac{1}{v}\right)^\alpha = v^\alpha \left(1 + \frac{\alpha}{v} + O(v^{-2})\right).
\end{equation*}
Then
\begin{equation*}
(v+1)^\alpha - v^\alpha = \alpha v^{\alpha-1} + O(v^{\alpha-2}).
\end{equation*}
Hence,
\begin{equation*}
F(v) \sim |\alpha|^{1/\beta} \, |v|^{\frac{\alpha-1}{\beta}} = C |v|^{\frac{2-2\beta}{\beta}} = C |v|^{\frac{2}{\beta} - 2}.
\end{equation*}
The integral at \(\pm\infty\) converges if the exponent is less than \(-1\):
\begin{equation*}
\frac{2}{\beta} - 2 < -1 \;\Longrightarrow\; \frac{2}{\beta} < 1 \;\Longrightarrow\; \beta > 2.
\end{equation*}
Thus, for \(\beta \leq 2\) it diverges, while for \(\beta>2\) it converges. Collecting the conditions:
\begin{itemize}
    \item Convergence at \(v = 0\) and \(v = -1\): \(\beta < 3\) (strict).
    \item Convergence at \(\pm\infty\): \(\beta > 2\) (strict).
\end{itemize}
Consequently, the integral \(I\) is finite if and only if
\begin{equation}
\boxed{2 < \beta < 3.}
\end{equation}

\section*{Appendix B: proof of (\ref{tech2})}

We will show

\[
I=\int_{0}^{\infty}\left(\int_{0}^{\infty}\Bigl|\int_{0}^{1}\Bigl[\frac{v^2}{4(\tau+\theta)^2}-\frac{1}{2(\tau+\theta)}\Bigr]\Phi_{v}(\tau+\theta)\,d\theta\Bigr|^{\beta}d\tau\right)^{\frac{1}{\beta}}dz< \infty.
\]

Recalling that \(\Phi_{\tau}(v)=G_\tau(v)=\frac{1}{\sqrt{4\pi\tau}}\exp{\frac{-v^2}{4\tau}}\), and applying the triangle inequality followed by Jensen's inequality, we derive

\begin{equation}
I \leq \int_{0}^{\infty}\left(\int_{0}^{\infty}\int_{0}^{1}\frac{v^{2\beta}}{4^{\beta}(\tau+\theta)^{2\beta}}\Phi_{v}(\tau+\theta)^\beta\,d\theta d\tau +\int_0^{\infty}\int_{0}^{1}\frac{1}{2^{\beta}(\tau+\theta)^{\beta}}\Phi_{v}(\tau+\theta)^{\beta}\,d\theta\,d\tau\right)^{\frac{1}{\beta}}dz.
\end{equation}
We now define
\begin{align*}
    J_1(v)&=\int_{0}^{\infty}\int_{0}^{1}\frac{v^{2\beta}}{(\tau+\theta)^{2\beta}}\Phi_{v}(\tau+\theta)^\beta\,d\theta\, d\tau, \\
    J_2(v)&=\int_0^{\infty}\int_{0}^{1}\frac{1}{(\tau+\theta)^{\beta}}\Phi_{v}(\tau+\theta)^{\beta}\,d\theta\,d\tau.
\end{align*}
Interchanging the order of integration and then applying the change of variables \(a=\frac{v^2}{(\tau+\theta)}\) for both \(J_1\) and \(J_2\), analogously to the proof of the Apendix A, we obtain
\begin{align*}
    J_1(v)&=Cv^{2-3\beta}\int_{0}^{1}\int_{0}^{v^2/\theta}a^{\frac{5\beta}{2}-2}e^{-\frac{\beta a}{4}}\,da\, d\theta, \\
    J_2(v)&=Cv^{2-3\beta}\int_{0}^{1}\int_{0}^{v^2/\theta}a^{\frac{3\beta}{2}-2}e^{-\frac{\beta a}{4}}\,da\, d\theta.
\end{align*}
Recall that for any \(x,y >0\) and \(s>0\), \((x+y)^{s}\leq C_{s}(x^{s}+y^{s})\). In particular, for \(0<s\leq1\), we have \((x+y)^{s}\leq x^{s}+y^{s}\). Noting that both \(J_1\) and \(J_2\) are positive, we obtain:
\begin{equation}
   I\leq \int_{0}^{\infty}(J_1(v)+J_2(v))^{\frac{1}{\beta}}\, dv \leq \left[\int_{0}^{\infty}J_1(v)^{\frac{1}{\beta}}\, dv+\int_{0}^{\infty}J_2(v)^{\frac{1}{\beta}}\, dv\right].    
\end{equation}
Observe that as \(v\to\infty\), the integrals \(\int_{0}^{1}\int_{0}^{v^2/\theta}a^{\frac{5\beta}{2}-2}e^{-\frac{\beta a}{4}}\,da\, d\theta\) and \(\int_{0}^{1}\int_{0}^{v^2/\theta}a^{\frac{3\beta}{2}-2}e^{-\frac{\beta a}{4}}\,da\, d\theta\) converge provided \(\beta>1\). It thus suffices to analyze the common factor \(v^{2-3\beta}\) in \(J_1\) and \(J_2\). Raising this factor to the power \(1/\beta\) yields \(v^{\frac{2}{\beta}-3}\), whose exponent is less than \(-1\) if and only if \(\beta>1\). Hence, the integral converges as \(v\to\infty\). It remains to examine the behavior near the origin. For \(v\leq1\), consider the integrals
\begin{equation}
I_1:=\int_{0}^{1}\int_{0}^{v^2/\theta}a^{\frac{5\beta}{2}-2}e^{-\frac{\beta a}{4}}\,da\, d\theta \qquad\text{and}\qquad I_2:=\int_{0}^{1}\int_{0}^{v^2/\theta}a^{\frac{3\beta}{2}-2}e^{-\frac{\beta a}{4}}\,da\, d\theta.
\end{equation}
For the first integral, we have
\begin{align*}
    I_1&=\int_{0}^{1}\int_{0}^{v^2/\theta}a^{\frac{5\beta}{2}-2}e^{-\frac{\beta a}{4}}\,da\, d\theta \\
    &=\int_{0}^{v^2}\int_{0}^{v^2/\theta}a^{\frac{5\beta}{2}-2}e^{-\frac{\beta a}{4}}\,da\, d\theta + \int_{v^2}^{1}\int_{0}^{v^2/\theta}a^{\frac{5\beta}{2}-2}e^{-\frac{\beta a}{4}}\,da\, d\theta \\
    &\leq \int_{0}^{v^2}\int_{0}^{\infty}a^{\frac{5\beta}{2}-2}e^{-\frac{\beta a}{4}}\,da\, d\theta + \int_{v^2}^{1}\int_{0}^{v^2/\theta}a^{\frac{5\beta}{2}-2}e^{-\frac{\beta a}{4}}\,da\, d\theta.
\end{align*}
For the first summand, the integral with respect to \(a\) converges, leaving only \(v^2\). For the second summand, we interchange the order of integration and bound over the first rectangular region of integration. For \(a\to0^{+}\), we have \(1<v^2/a\), which allows us to write the integral as
\begin{equation}
\int_{v^2}^{1}\int_{0}^{v^2/\theta}da\,d\theta =\int_{0}^{v^2}\int_{v^2}^{1} d\theta\, da + \int_{v^2}^{1}\int_{v^2}^{v^2/a} d\theta\, da \leq \int_{0}^{1}\int_{v^2}^{v^2/a} d\theta\, da.
\end{equation}
Consequently,
\begin{equation*}
    I_1\leq C_{\beta}v^2 + \int_{0}^{1}\int_{v^2}^{v^2/a} a^{\frac{5\beta}{2}-2}e^{-\frac{\beta a}{4}}\,d\theta\, da.
\end{equation*}
Focusing on the second summand, we obtain
\begin{equation*}
\int_{0}^{1}\int_{v^2}^{v^2/a} a^{\frac{5\beta}{2}-2}e^{-\frac{\beta a}{4}}\,d\theta\, da = v^2\int_{0}^{1} (1-a)a^{\frac{5\beta}{2}-3}e^{-\frac{\beta a}{4}}\, da\leq C_{\beta}v^2,
\end{equation*}
where the integral \(\int_{0}^{1} (1-a)a^{\frac{5\beta}{2}-3}e^{-\frac{\beta a}{4}}\, da\) converges if and only if \(\frac{5\beta}{2}-3>-1 \Longleftrightarrow \beta>4/5\). Hence,
\begin{equation*}
I_1\leq C_{\beta}v^2.
\end{equation*}
For the second integral \(I_2\), an entirely analogous argument shows that convergence requires the integral \(\int_{0}^{1} (1-a)a^{\frac{3\beta}{2}-3}e^{-\frac{\beta a}{4}}\, da\) to converge, which occurs if and only if \(\frac{3\beta}{2}-3>-1 \Longleftrightarrow \beta>4/3\). Therefore, for \(\beta>4/3\), we have
\begin{align*}
    I &\leq \left[\int_{0}^{\infty}J_1(v)^{\frac{1}{\beta}}\, dv+\int_{0}^{\infty}J_2(v)^{\frac{1}{\beta}}\, dv\right] \\
    &= \biggl[\int_{1}^{\infty}J_1(v)^{\frac{1}{\beta}}\, dv+\int_{1}^{\infty}J_2(v)^{\frac{1}{\beta}}\, dv + \int_{0}^{1}J_1(v)^{\frac{1}{\beta}}\, dv+\int_{0}^{1}J_2(v)^{\frac{1}{\beta}}\, dv\biggr] \\
    &\leq \left[C_{\beta}+C_{\beta}\int_{0}^{1}v^{\frac{4}{\beta}-3}\right],
\end{align*}
where the last integral converges if and only if \(\frac{4}{\beta}-3>-1 \Longleftrightarrow \beta <2\). We thus conclude that
\begin{equation}
    I < \infty
\end{equation}
for all \(\beta\in(\frac{4}{3},2)\).

\section*{Data availability statement}

Data sharing is not applicable to this article as no data sets 
were generated or analysed during the current study.

 \section*{Acknowledgements}

The author Christian Olivera is partially supported by  by FAPESP-ANR by the grant Stochastic and Deterministic Analysis for Irregular Models$-2022/03379-0$  and CNPq by the grant $422145/2023-8$.  C. Tudor  acknowledges support from   the  ANR project SDAIM 22-CE40-0015, the MATHAMSUD grant 24-MATH-04 SDE-EXPLORE,  and  by the Ministry of Research, Innovation and Digitalization (Romania), grant CF-194-PNRR-III-C9-2023. C. Tudor also acknowledges the support of the CDP C2EMPI, together with the French State under the France-2030 programme, the University of Lille, the Initiative of Excellence of the University of Lille, the European Metropolis of Lille for their funding and support of the R-CDP-24-004-C2EMPI project.   The author Simon Chony Acosta is supported by the fellowship FAPESP $2024/06348-3$.

 \section*{Conflict of Interest}
There is no conflict of interests.

\end{document}